\documentclass[11pt,a4paper]{article}
\usepackage{hyperref}
\usepackage[english]{babel}
\usepackage[utf8]{inputenc}
\usepackage[T1]{fontenc}
\usepackage{amsmath,amsthm,amssymb,amsfonts}
\numberwithin{equation}{section}
\usepackage{enumerate}
\usepackage{faktor}
\usepackage{dsfont}
\usepackage{scalerel,stackengine}
\usepackage{textcomp}
\usepackage{graphicx}
\usepackage{extarrows}
\usepackage{epigraph}
\usepackage{graphicx}
\usepackage{subcaption}
\usepackage{sidecap}
\usepackage{indentfirst}
\usepackage{tikz}
\usepackage{tikz-cd}
\usetikzlibrary{shapes.misc,arrows,matrix,patterns,decorations.markings,positioning}
\tikzset{cross/.style={cross out, draw=black, minimum size=2*(#1-\pgflinewidth), inner sep=0pt, outer sep=0pt},
cross/.default={4.5pt}}
\usepackage{pgfplots}
\usepackage{caption}
\usepackage{float}
\usepackage{color}
\usepackage{epstopdf}
\usepackage{epsfig}
\usepackage{stmaryrd}
\usepackage{color}
\usepackage{geometry}
\usepackage{multicol}
\usepackage{verbatim}

\DeclareMathOperator{\lk}{\ell k}

\DeclareMathOperator{\Spin}{Spin}

\DeclareMathOperator{\rk}{rk}

\renewcommand{\geq}{\geqslant}
\renewcommand{\leq}{\leqslant} 
\renewcommand{\epsilon}{\varepsilon}

\newcommand{\Z}{\mathbb{Z}}

\newcommand{\F}{\mathbb{F}}

\DeclareFontFamily{U}{mathx}{\hyphenchar\font45}
\DeclareFontShape{U}{mathx}{m}{n}{
      <5> <6> <7> <8> <9> <10>
      <10.95> <12> <14.4> <17.28> <20.74> <24.88>
      mathx10
      }{}
\DeclareSymbolFont{mathx}{U}{mathx}{m}{n}
\DeclareFontSubstitution{U}{mathx}{m}{n}
\DeclareMathAccent{\widecheck}{0}{mathx}{"71}
\DeclareMathAccent{\wideparen}{0}{mathx}{"75}

\newtheorem{teo}{Theorem}[section]
\newtheorem*{teo*}{Theorem}
\newtheorem{lemma}[teo]{Lemma}
\newtheorem{prop}[teo]{Proposition}
\newtheorem*{prop*}{Proposition}

\usepackage{xpatch}
\makeatletter
\xpatchcmd{\@thm}{\thm@headpunct{.}}{\thm@headpunct{}}{}{}
\makeatother

\theoremstyle{definition}
\newtheorem{remark}[teo]{Remark}

\pgfplotsset{compat=1.13}
\begin{document}
\title{Nearly fibered links with genus one}
\author{Alberto Cavallo and Irena Matkovi\v{c}\\ \\
 \footnotesize{Universit\'e du Qu\'ebec \`a Montr\'eal (UQAM),}\\ 
 \footnotesize{Montr\'eal (QC) H3C 3P8, Canada}\\ \\
 \footnotesize{Uppsala Universitet,}\\
 \footnotesize{Uppsala 751 06, Sweden}\\ \\  \small{alberto.cavallo@mcgill.ca}\hspace{2cm}\small{irena.matkovic@math.uu.se}}
\date{}

\maketitle
\begin{abstract}
 We classify all the $n$-component links in the $3$-sphere that bound a Thurston norm minimizing Seifert surface $\Sigma$ with Euler characteristic $\chi(\Sigma)=n-2$ and that are nearly fibered, which means that the rank of their link Floer homology group $\widehat{HFL}$ in the maximal (collapsed) Alexander grading $s_{\text{top}}$ is equal to two. In other words, such a link $L$ satisfies $s_{\text{top}}=\frac{n-\chi(\Sigma)}{2}=1$, and in addition $\rk\widehat{HFL}_{*}(L)[1]=2$ and $\rk\widehat{HFL}_{*}(L)[s]=0$ for every $s>1$. 
 
 The proof of the main theorem is inspired by the one of a similar recent result for knots by Baldwin and Sivek, and involves techniques from sutured Floer homology. Furthermore, we also compute the group $\widehat{HFL}$ for each of these links. 
\end{abstract}

\section{Introduction}
The \emph{Seifert genus} $g_3$ is one of the most important invariants of knots. Its definition is standard in literature: say $\Sigma$ is a \emph{Seifert surface} for a knot $K$ if it is a compact, oriented surface, embedded in $S^3$, such that $\partial\Sigma=K$; we call $g_3(K)$ the minimal genus of a connected Seifert surface for $K$.
When we want to study a link $L$ with $n>1$ components this is no longer true, but there are two main possibilities:
\begin{itemize}
    \item either we can consider $g_3(L)$, using precisely the same definition as above;
    \item or we take \[G_3(L)=\dfrac{n-\chi_3(F)}{2}\:,\] which we call the \emph{Seifert big genus} of $L$, and we define $\chi_3$ as the maximal Euler characteristic of any Seifert surface for $L$ (dropping the requirement of $F$ being connected).
\end{itemize}
We recall that in $S^3$ we can define \emph{fibered links} as the links whose complements fiber over the circle; then the following is a classical result in low dimensional topology.
\begin{teo}[Stallings \cite{Stallings}]
 \label{teo:classic}
 The only fibered links in $S^3$, with Seifert big genus equal to one, are the trefoil knots $T_{2,\pm3}$, the figure-eight knot $4_1$ and the Hopf links $T_{2,\pm2}$. 
\end{teo}
We recall that the link Floer chain complex $\left(\widehat{CFL}_*(\mathcal D),\widehat\partial\right)$ over the field with two elements $\F$, introduced by Ozsv\'ath and Szab\'o in \cite{Multivariable}, can be constructed from a Heegaard diagram $\mathcal D=(\Sigma,\alpha,\beta,\textbf w,\textbf z)$, where the sets of basepoints both contain $n$ elements. Under some conditions, the diagram $\mathcal D$ represents an $n$-component link $L$.
If we ignore the information given by $\textbf z$ then $\mathcal D$ is just a (multi-pointed) Heegaard diagram for $S^3$ and then $\left(\widehat{CFL}_*(\mathcal D),\widehat\partial\right)$ can be interpreted as a multi-filtered complex (the filtration is induced by the link $L$) whose total homology is $\left(\F_{-1}\oplus\F_0\right)^{\otimes\:n-1}$. The absolute grading $*$ is called \emph{Maslov grading}, while the multi-filtration the \emph{Alexander filtration}. 

We write $\widehat{\textbf{HFL}}_{*}(L)[s_1,...,s_n]$ for the homology of the graded object associated to $\left(\widehat{CFL}_*(\mathcal D),\widehat\partial\right)$.
This is a multi-graded homology group, where now the Alexander filtration is turned into an absolute (multi-) grading $[s_1,...,s_n]$. Such a group is an isotopy invariant of $L$ in $S^3$. We can define a slightly different homology group by collapsing the Alexander multi-grading as follows 
\begin{equation}
\widehat{HFL}_*(L)[s]=\bigoplus_{s=s_1+...+s_n}\widehat{\textbf{HFL}}_{*}(L)[s_1,...,s_n]\:.
\label{eq:collapsed}
\end{equation}
It is proved in \cite{Ni2} that 
\begin{equation}
\label{eq:genus}
 s_{\text{top}}:=\max_{s\in\Z}\left\{\rk\widehat{HFL}_*(L)[s]\neq0\right\}=G_3(L)\:,
\end{equation}
while
a very important result of Ghiggini \cite{Ghiggini} and Ni \cite{Ni} is that a link is fibered if and only if its link Floer homology satisfies the following property:
\begin{equation}
 \label{eq:rank}
 \rk\widehat{HFL}_*(L)[s_\text{top}]=1\:.
\end{equation}
Combining Equations \eqref{eq:genus} and \eqref{eq:rank}, we obtain that link Floer homology detects each of the five links mentioned in Theorem \ref{teo:classic}; in other words, those are the only links such that $G_3(L)=1$ and $\widehat{HFL}_*(L)[1]\cong\F$.

In this paper we want to take a step forward and classify all the links in $S^3$ such that $G_3(L)=1$, but this time with $\widehat{HFL}_*(L)[1]\cong\F^2$; we call such links \textbf{nearly fibered} following the terminology in \cite{BS}. The first part of this classification is the case of knots which was done recently by Baldwin and Sivek.
\begin{teo}[Baldwin-Sivek \cite{BS}]
 \label{teo:BS}
 A nearly fibered knot $K$ such that $g_3(K)=1$ is isotopic to either the knot $5_2$, the knot $15_{43522}$, the $2$-framed Whitehead doubles of the trefoil knot $\emph{Wh}^\pm_2(T_{2,3})$ or a pretzel knot $P(-3,3,2m+1)$ for $m\geq0$ (including the Stevedore knot $6_1$) up to mirroring.
\end{teo}
Our main result extends the classification in Theorem \ref{teo:BS} to links.
\begin{figure}[ht]
 \centering
 \includegraphics[width=6cm]{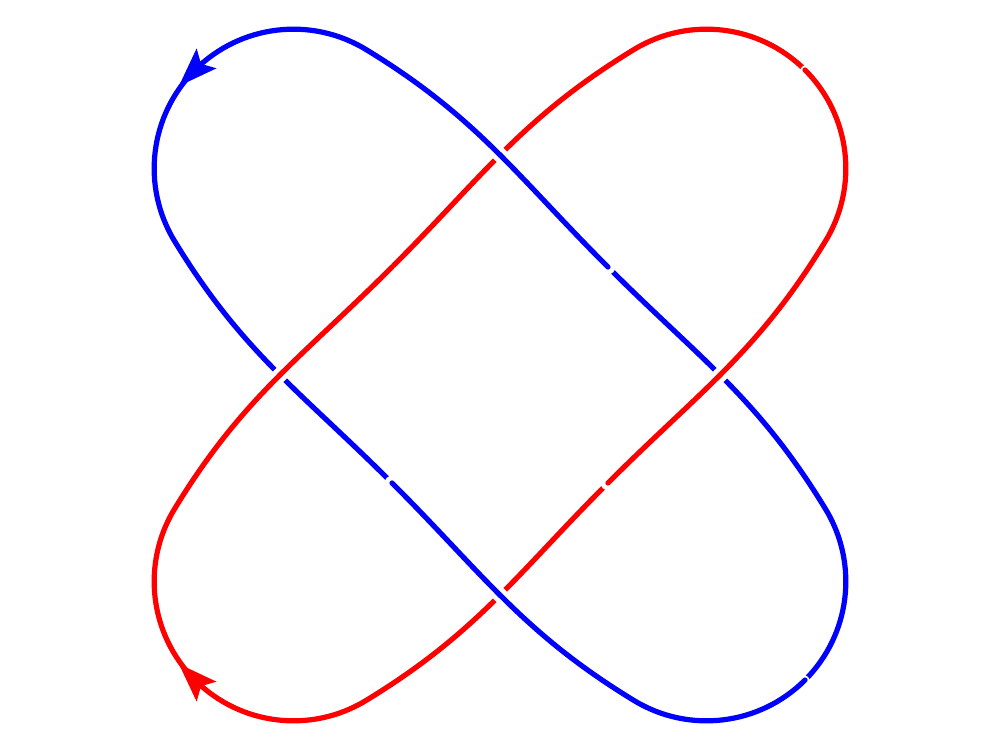}
 \hspace{1cm}
 \includegraphics[width=7cm]{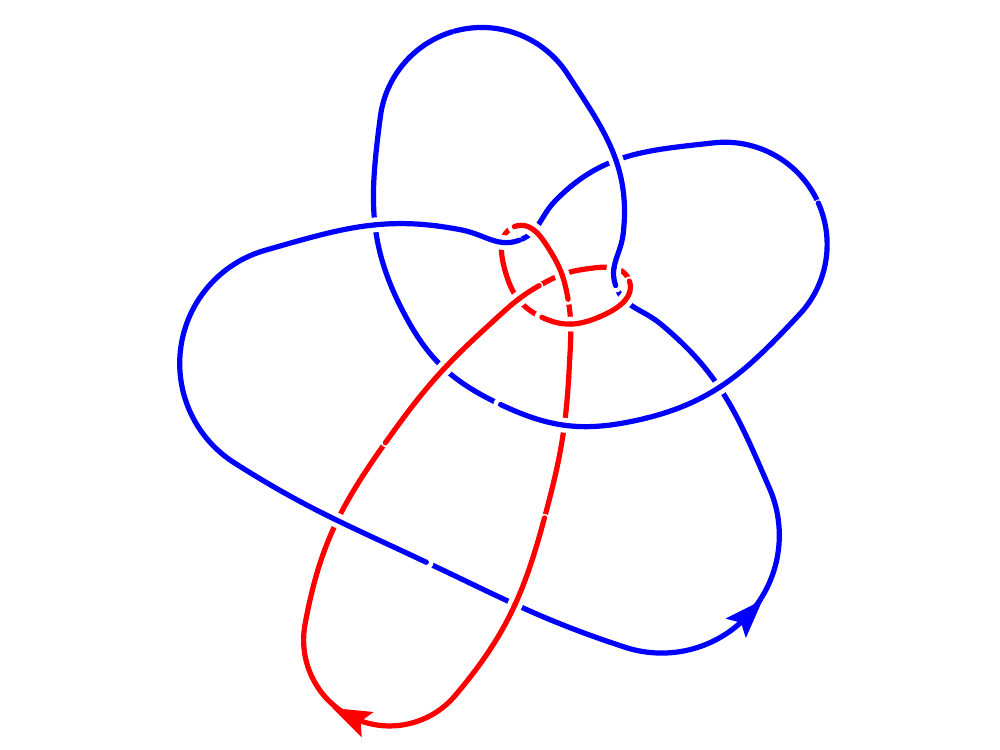}
 \caption{A diagram for the links $T^{'}_{2,4}$ (left) and $T^{'}_{2,3;2,4}$ (right). Both links bound an embedded annulus in $S^3$.}
 \label{B-W}
\end{figure}
\begin{teo}
\label{teo:main}
 A nearly fibered link $L$, with $n>1$ components, such that $G_3(L)=1$, is isotopic to either $K\sqcup\bigcirc$, where $K$ is a trefoil knot, a figure-eight knot or a Hopf link, the link $T_{2,4}$ or the $(2,4)$-cable of $T_{2,3}$, with oppositely oriented components, up to mirroring. 
\end{teo}
The last two links are pictured in Figure \ref{B-W}. From now on we use the notation that $T^{'}_{2,4}$ represents the torus link $T_{2,4}$ and $T^{'}_{2,3;2,4}$ represents the $(2,4)$-cable of the positive trefoil knot, both with oppositely oriented components.

We can compute the link Floer homology group $\widehat{HFL}(L)$ for every link in Theorem \ref{teo:main}; and then also of their mirror images using the formula 
\begin{equation}
 \widehat{HFL}_d(L^*)[s]\cong\widehat{HFL}_{-d+1-n}(L)[-s]
 \label{eq:mirror}
\end{equation}
for every $d,s\in\Z$, which is proved in \cite{Book}.
\begin{prop}
 \label{prop:cable}
 We have that \[\widehat{HFL}_d(T^{'}_{2,3;2,4})[s]\cong\F^2_{(-1)}[1]\oplus\F^4_{(-2)}[0]\oplus\F_{(-1)}[0]\oplus\F_{(0)}[0]\oplus\F^2_{(-3)}[-1]\:.\]
\end{prop}
As far as we know this is the first time that this homology group has been computed completely.
Furthermore, we obtain that $\widehat{HFL}(L)$ is different for every link in Theorems \ref{teo:BS} and \ref{teo:main} and this implies that link Floer homology detects each nearly fibered link of Seifert big genus one, confirming some results of Binns and Dey in \cite{BD} and Binns and Martin in \cite{BM}.

The paper is organized as follows: in Section 2 we prove Theorem \ref{teo:main}, separating the case of split and non-split links. Finally, in Section 3 we present our table of genus one nearly fibered links and we explain the proof of Proposition \ref{prop:cable}.

\paragraph*{Acknowledgements} AC has a CIRGET post-doctoral fellowship at the Universit\'e du Qu\'ebec \`a Montr\'eal. IM is supported by the Knut and Alice Wallenberg Foundation through the grant KAW 2021.0191, and by the Swedish Research Counsil through the grant number 2020-04426, as a postdoctoral researcher at Uppsala Universitet. We are both grateful to the Alfr\'ed R\'enyi Institute of Mathematics in Budapest for its hospitality during the semester on "Singularities and low-dimensional topology", and for the support from the \'Elvonal (Frontier) grant KKP126683 (given by NKFIH). We thank Steven Sivek for a helpful conversation, and the referee for their corrections. 

\section{Proof of the main result}
\subsection{Split links}
Suppose that $L$ is as in Theorem \ref{teo:main} and is a split link; in this case, we can write $L=L_1\sqcup L_2$ for some links $L_1$ and $L_2$. We recall that the link Floer homology of a disjoint union, see \cite{Multivariable,Cavallo}, satisfies the following relation 
\begin{equation}
 \begin{aligned}
  \widehat{HFL}_{d}(L_1\sqcup L_2)[s]\cong\bigoplus_{\substack{d_1+d_2=d\\s_1+s_2=s}}\widehat{HFL}_{d_1}(L_1)[s_1]&\otimes\widehat{HFL}_{d_2}(L_2)[s_2]\:\oplus \\ &\oplus\:\bigoplus_{\substack{d_1+d_2=d+1\\s_1+s_2=s}}\widehat{HFL}_{d_1}(L_1)[s_1]\otimes\widehat{HFL}_{d_2}(L_2)[s_2]\:.
 \end{aligned} 
 \label{eq:disjoint}
\end{equation}
This implies that \[2=\rk\widehat{HFL}_{*}(L)[1]=2\cdot\rk\widehat{HFL}_{*}(L_1)[G_3(L_1)]\cdot\rk\widehat{HFL}_{*}(L_2)[G_3(L_2)]\:;\] since the top-Alexander grading rank cannot be zero by definition, we have that \[\rk\widehat{HFL}_{*}(L_i)[G_3(L_i)]=1\] and then $L_i$ is fibered for $i=1,2$, see \cite{Ghiggini,Ni}.
Moreover, it has to be the case that $G_3(L_1)=1$ and $G_3(L_2)=0$ (up to relabelling); and a big genus one fibered link is either a trefoil or a figure-eight knot or a Hopf link from Theorem \ref{teo:classic}, while it is a classical result that the only big genus zero fibered link is the unknot. This proves the first part of Theorem \ref{teo:main}.

\subsection{Non-split links}
The first step in this case is the following lemma.
\begin{lemma}
 \label{lemma:2}
 A non-split link $L$ such that $G_3(L)=1$ has at most two components.
\end{lemma}
\begin{proof}
 Say $F\hookrightarrow S^3$ is a surface that realizes the Seifert big genus $G_3$. A quick computation shows that $1=G_3(L)=g(F)+n-k$, where $n$ is the number of components of $L$ and $k$ the number of connected components of $F$. Hence, since $g(F)$ and $n-k$ are both non-negative, we have that either \[\left\{\begin{aligned}
 &g(F)=1 \\
 &n=k 
 \end{aligned}\right.\hspace{1cm}\text{ or }\hspace{1cm}\left\{\begin{aligned}
 &g(F)=0 \\
 &n-1=k 
 \end{aligned}\right.\:.\] In the first case $F$ has $n$ connected components $F_1,...,F_n$, each one with a knot as boundary; this implies that $F_2,...,F_n$ are all diffeomorphic to disks (up to relabelling) and then $L$ has $n-1$ unlinked unknotted components. This is a contradiction, since $L$ is non-split, unless $n=1$.

 In the second case $F$ has $n-1$ connected components $F_1,...,F_{n-1}$, each one with a knot as boundary except one of them which is an annulus; this implies that $F_2,...,F_{n-1}$ are again diffeomorphic to disks (again up to relabelling) and, in the same way as before, we obtain that $n=2$.
\end{proof}

Here, we pass to sutured Floer homology, as defined by Juh\'asz in \cite{J}. We recall that a \emph{balanced sutured manifold} is a pair $(M, \gamma)$, where $M$ is a compact oriented 3-manifold with boundary and with no closed components, and the sutures $\gamma\subset\partial M$ are oriented $1$-manifolds, meeting each boundary component, that divide the boundary into two subsurfaces $R_+$ and $R_-$, such that $\chi(R_+)=\chi(R_-)$. Similarly to the link case, balanced sutured manifolds can be presented by multi-pointed Heegaard diagrams, and then sutured Floer homology is constructed in the same way and as a generalization of the hat version of Heegaard Floer homology; it assigns to $(M,\gamma)$ a finitely generated vector space $SFH(M,\gamma)$ which splits along relative $\Spin^c$ structures. The main advantage of the sutured Floer viewpoint is its neat behavior under sutured manifold decomposition.

In the case of links, we refer to a link complement with a pair of (oppositely oriented)  meridional sutures on each boundary component as the sutured link complement $S^3(L):=\left(S^3\setminus\mathring{\nu(L)}, \gamma_\mu\right)$. This way we have an isomorphism $SFH(S^3(L))\cong\widehat{\textbf{HFL}}(L)$ where the Alexander multi-grading is recovered by evaluations of relative $\Spin^c$ structures on particular surfaces bounded by components of $L$; in particular, for the collapsed Alexander grading we use a Seifert surface of $L$. Furthermore, looking at the sutured Seifert surface complement $S^3(F)=(S^3\setminus\mathring{\nu(F)},\gamma_F):=\left(S^3\setminus F\times(-1,1),\partial F\times\{0\}\right)$, the corresponding subspace is exactly the top Alexander grading summand $SFH(S^3(F))\cong\widehat{HFL}(L)[G_3(L)]$. 

Returning to our analysis of non-split links, we have from Lemma \ref{lemma:2} that $L$ bounds an annulus $F$ in $S^3$, and by assumption $SFH(S^3(F))\cong\widehat{HFL}(L)[G_3(L)]$ is of rank 2. Now, to understand our $(S^3\setminus\mathring{\nu(F)},\gamma_F)$ we follow the argument of Baldwin and Sivek from \cite[Section 5]{BS}, which, in turn, relies on the properties of link complements with rank-2 sutured Floer homology \cite{Juhasz} and some classical facts about $L$-space knots. We get, as in \cite[Theorem 5.1]{BS}, that up to reversing the orientation $(S^3\setminus\mathring{\nu(F)},\gamma_F)$, whose boundary is a torus, is identified with $(S^3\setminus\mathring{\nu(C)},\gamma_2)$, where $C$ is either the unknot $\bigcirc$ or the positive trefoil $T_{2,3}$ and the suture $\gamma_2$ is given by two parallel simple closed curves on $\partial\nu(F)$ oppositely oriented and with slope 2.

The idea now is to show that $L$ is actually isotopic to $C^\prime_{2,4}$, the $(2,4)$-cable of $C$ with oppositely oriented components, up to mirroring. This will complete the proof of Theorem \ref{teo:main}. 
\begin{prop}
 \label{prop:complement}
 Up to reversing the orientation, the complement $S^3\setminus\mathring{\nu(L)}$ is diffeomorphic to $S^3\setminus\mathring{\nu(C^\prime_{2,4})}$. Furthermore, there is such a diffeomorphism $\Phi$ for which $\Phi(\lambda_i)=\lambda_i^\prime$ with $i=1,2$, where $\lambda_i$ and $\lambda_i^\prime$ are longitudes for $L$ and $C^\prime_{2,4}$ respectively.
\end{prop}
\begin{proof}
 The manifold $S^3\setminus\mathring{\nu(L)}$ is obtained by gluing up the thickened annulus $A\times[-1,1]$ along $A\times \{-1,1\}$ to $S^3\setminus\mathring{\nu(F)}$. Such a gluing is determined by the attaching regions: the two parallel annuli complementary to a neighborhood of the sutures $\mathring{\nu(\gamma_F)}$. 
 Therefore, using again \cite[Theorem 5.1]{BS} we get that $S^3\setminus\mathring{\nu(L)}$ is diffeomorphic to the 3-manifold given by attaching $A\times[-1,1]$ to $S^3\setminus\mathring{\nu(C)}$ along  $\partial\nu(C)\setminus\mathring{\nu(\gamma_2)}$, which is precisely $S^3\setminus\mathring{\nu(C^\prime_{2,4})}$. We denote this diffeomorphism by $\Phi$ and we observe that $\Phi(A)$ is an annulus whose boundary is $C_{2,4}^\prime$.

 For the second part of the statement, we note that by the definition of longitude one has $\lambda_1\sqcup\lambda_2=\partial A$; and also $\lambda^\prime_1\sqcup\lambda^\prime_2=\partial \Phi(A)$ because $\Phi$ is orientation preserving. Hence, we can write
 \[\lambda^\prime_1\sqcup\lambda^\prime_2=\partial \Phi(A)=\Phi(\partial A)=\Phi(\lambda_1\sqcup\lambda_2)=\Phi(\lambda_1)\sqcup \Phi(\lambda_2)\] which implies $\Phi(\lambda_i)=\lambda_i^\prime$ for $i=1,2$ 
\end{proof}
\begin{remark}
 It is known that the Gordon-Luecke theorem does not hold for links: there is an infinite family of links, obtained by performing annular twisting on the Whitehead link, which have diffeomorphic complements but are all non-isotopic. Proposition \ref{prop:complement} is then by itself not enough to show that $L$ is isotopic to $C_{2,4}^\prime$. Since we are considering a link $L$ which bounds an embedded annulus in $S^3$, this could also be proved using the topology of the cabling construction; nonetheless, in this paper we decided to include a more direct proof.
 

\end{remark}
Our strategy continues as follows: we are going to prove that the diffeomorphism $\Phi$ from Proposition \ref{prop:complement} preserves the peripheral system of $L$; and since the latter is known to be a complete isotopy invariant of links \cite{W}, this is enough to prove Theorem \ref{teo:main}. We then need the following two lemmas.
\begin{lemma}
 \label{lemma:1}
 A link $L$ whose complement is diffeomorphic to the one of $C_{2,4}'$, through a diffeomorphism $\Phi$ which preserves the longitudes, has linking number $\lk(L)$ equal to $\pm2$.   
\end{lemma}
\begin{proof}
 Since $\Phi$ preserves the longitudes $\lambda_i$, we have that $L$ and $C_{2,4}$ have diffeomorphic $0$-surgeries (with respect to the framings given by $\lambda_i$ and $\Phi(\lambda_i))$. Then one has \[\lk^2(L)=|H_1(S^3_{0,0}(L);\Z)|=|H_1(S^3_{0,0}(C_{2,4});\Z)|=4\:.\]
\end{proof}
Let us denote the meridians of $L$ and $C_{2,4}^\prime$ by $\mu_i$ and $\mu_i'$ respectively for $i=1,2$. Since $\Phi$ is an orientation preserving  diffeomorphism we should have \[\left\{\begin{aligned}
&\Phi(\mu_1)=\mu_1'+p\cdot\lambda'_1 \\
&\Phi(\mu_2)=\mu_2'+q\cdot\lambda'_2\end{aligned}\right.\]
for some $p,q\in\Z$.
\begin{figure}[ht]
 \centering
 \begin{subfigure}[b]{0.3\textwidth}
 \def\svgwidth{10.5cm}
\begingroup%
  \makeatletter%
  \providecommand\color[2][]{%
    \errmessage{(Inkscape) Color is used for the text in Inkscape, but the package 'color.sty' is not loaded}%
    \renewcommand\color[2][]{}%
  }%
  \providecommand\transparent[1]{%
    \errmessage{(Inkscape) Transparency is used (non-zero) for the text in Inkscape, but the package 'transparent.sty' is not loaded}%
    \renewcommand\transparent[1]{}%
  }%
  \providecommand\rotatebox[2]{#2}%
  \newcommand*\fsize{\dimexpr\f@size pt\relax}%
  \newcommand*\lineheight[1]{\fontsize{\fsize}{#1\fsize}\selectfont}%
  \ifx\svgwidth\undefined%
    \setlength{\unitlength}{1450.68595684bp}%
    \ifx\svgscale\undefined%
      \relax%
    \else%
      \setlength{\unitlength}{\unitlength * \real{\svgscale}}%
    \fi%
  \else%
    \setlength{\unitlength}{\svgwidth}%
  \fi%
  \global\let\svgwidth\undefined%
  \global\let\svgscale\undefined%
  \makeatother%
  \begin{picture}(1,0.58759659)%
    \lineheight{1}%
    \setlength\tabcolsep{0pt}%
    \put(0,0){\includegraphics[width=\unitlength,page=1]{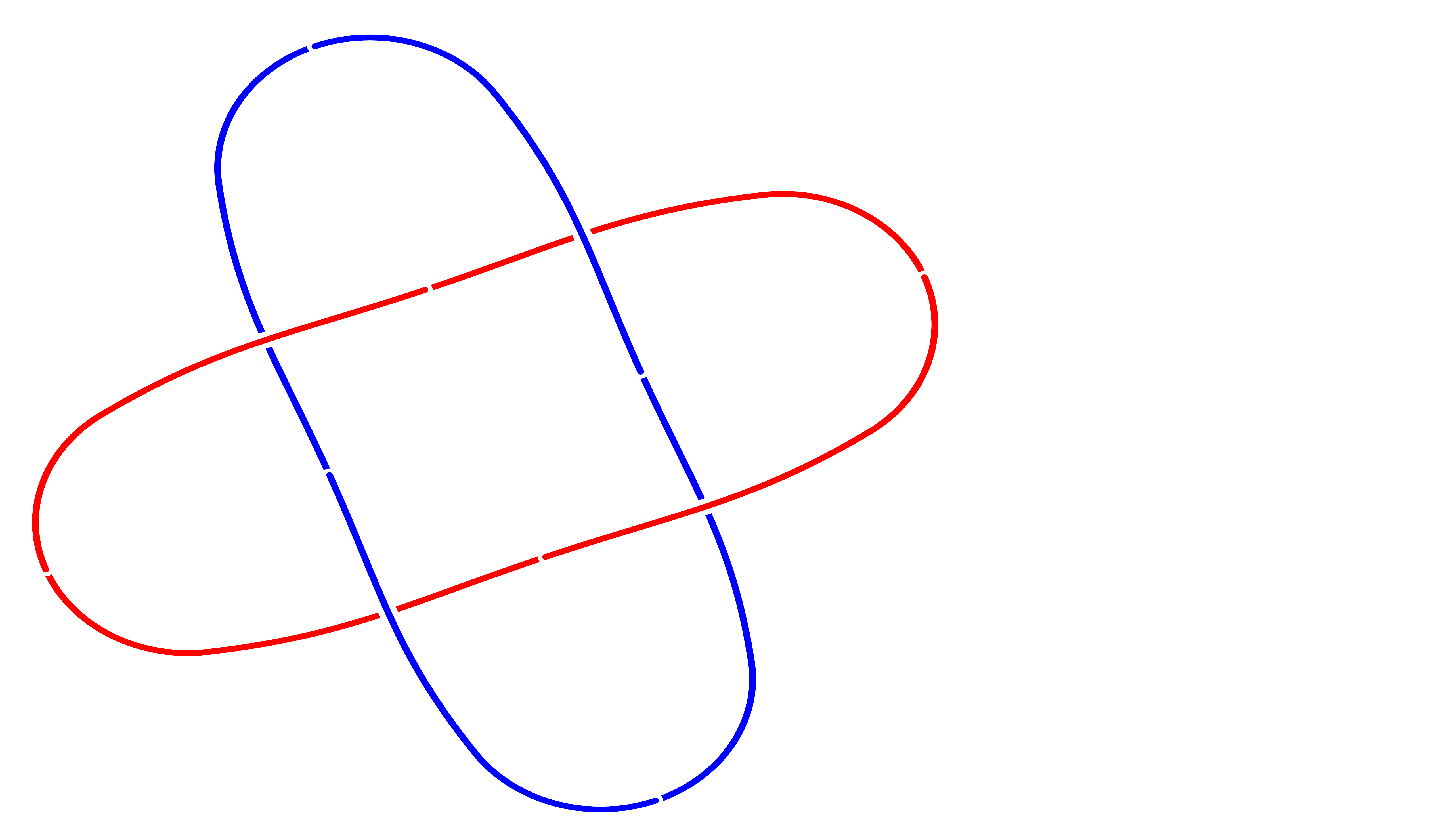}}%
    \put(0.37830519,0.55576898){\color[rgb]{0,0,0}\makebox(0,0)[lt]{\lineheight{1.25}\smash{\begin{tabular}[t]{l}$\frac{1}{q+1}$\end{tabular}}}}%
    \put(0.63616351,0.24259154){\color[rgb]{0,0,0}\makebox(0,0)[lt]{\lineheight{1.25}\smash{\begin{tabular}[t]{l}$\frac{1}{p+1}$\end{tabular}}}}%
  \end{picture}%
\endgroup%
   
 \end{subfigure}
 \hspace{3cm}
 \begin{subfigure}[b]{0.3\textwidth}
 \centering
 \def\svgwidth{7.5cm}
\begingroup%
  \makeatletter%
  \providecommand\color[2][]{%
    \errmessage{(Inkscape) Color is used for the text in Inkscape, but the package 'color.sty' is not loaded}%
    \renewcommand\color[2][]{}%
  }%
  \providecommand\transparent[1]{%
    \errmessage{(Inkscape) Transparency is used (non-zero) for the text in Inkscape, but the package 'transparent.sty' is not loaded}%
    \renewcommand\transparent[1]{}%
  }%
  \providecommand\rotatebox[2]{#2}%
  \newcommand*\fsize{\dimexpr\f@size pt\relax}%
  \newcommand*\lineheight[1]{\fontsize{\fsize}{#1\fsize}\selectfont}%
  \ifx\svgwidth\undefined%
    \setlength{\unitlength}{950.20580319bp}%
    \ifx\svgscale\undefined%
      \relax%
    \else%
      \setlength{\unitlength}{\unitlength * \real{\svgscale}}%
    \fi%
  \else%
    \setlength{\unitlength}{\svgwidth}%
  \fi%
  \global\let\svgwidth\undefined%
  \global\let\svgscale\undefined%
  \makeatother%
  \begin{picture}(1,0.86052951)%
    \lineheight{1}%
    \setlength\tabcolsep{0pt}%
    \put(0,0){\includegraphics[width=\unitlength,page=1]{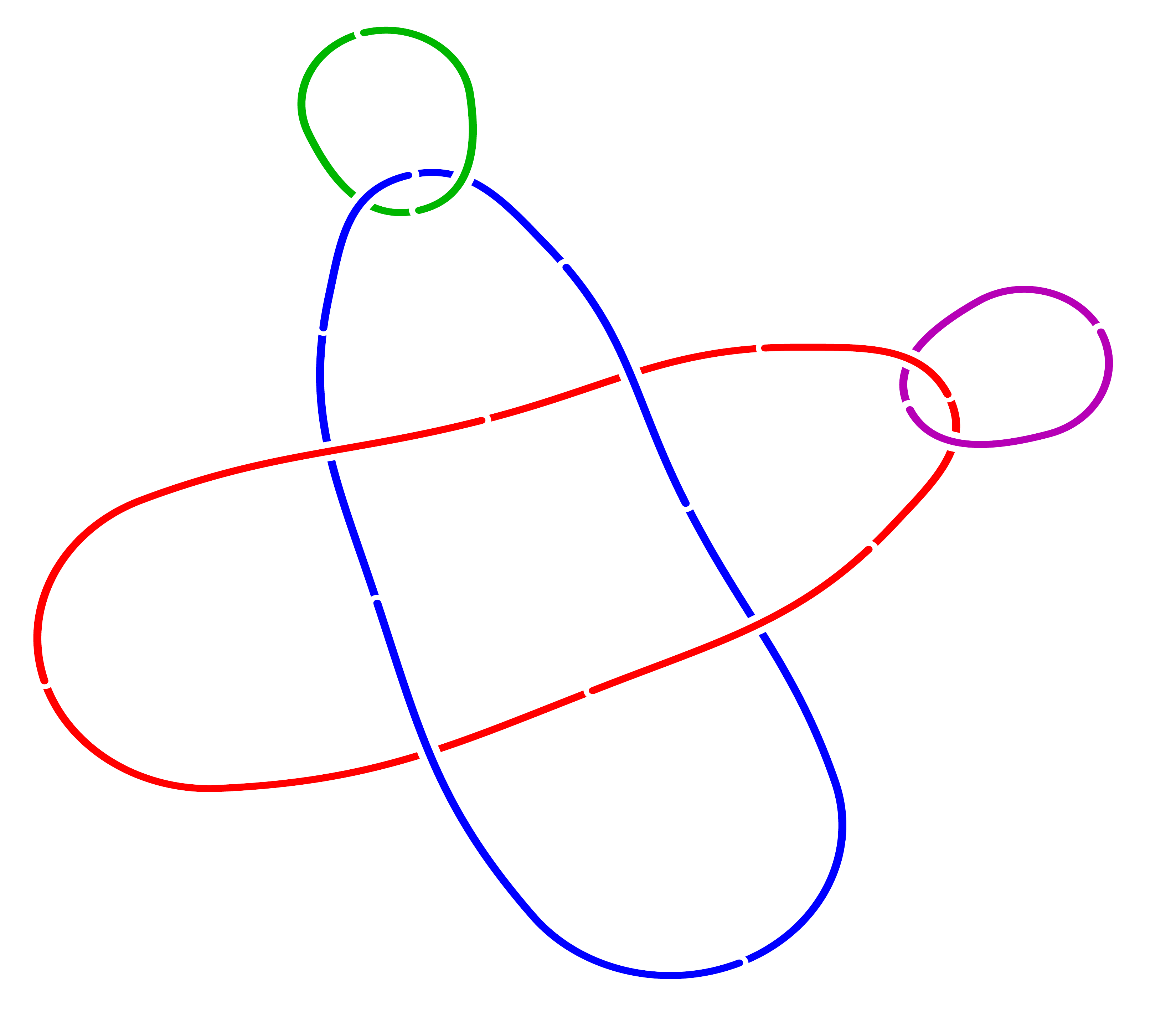}}%
    \put(0.08530152,0.28585313){\color[rgb]{0,0,0}\makebox(0,0)[lt]{\lineheight{1.25}\smash{\begin{tabular}[t]{l}$0$\end{tabular}}}}%
    \put(0.54708542,0.08424838){\color[rgb]{0,0,0}\makebox(0,0)[lt]{\lineheight{1.25}\smash{\begin{tabular}[t]{l}$0$\end{tabular}}}}%
    \put(0.75550116,0.67544068){\color[rgb]{0,0,0}\makebox(0,0)[lt]{\lineheight{1.25}\smash{\begin{tabular}[t]{l}$-(p+1)$\end{tabular}}}}%
    \put(0.04443568,0.75308576){\color[rgb]{0,0,0}\makebox(0,0)[lt]{\lineheight{1.25}\smash{\begin{tabular}[t]{l}$-(q+1)$\end{tabular}}}}%
  \end{picture}%
\endgroup%

 \end{subfigure}
 \caption{Slam-dunks performed on the link $C_{2,4}$, here the orientation is not relevant. One has that the 3-manifold represented by the surgery diagram on the left is diffeomorphic to the one on the right.}
 \label{Slamdunk}
 \end{figure}
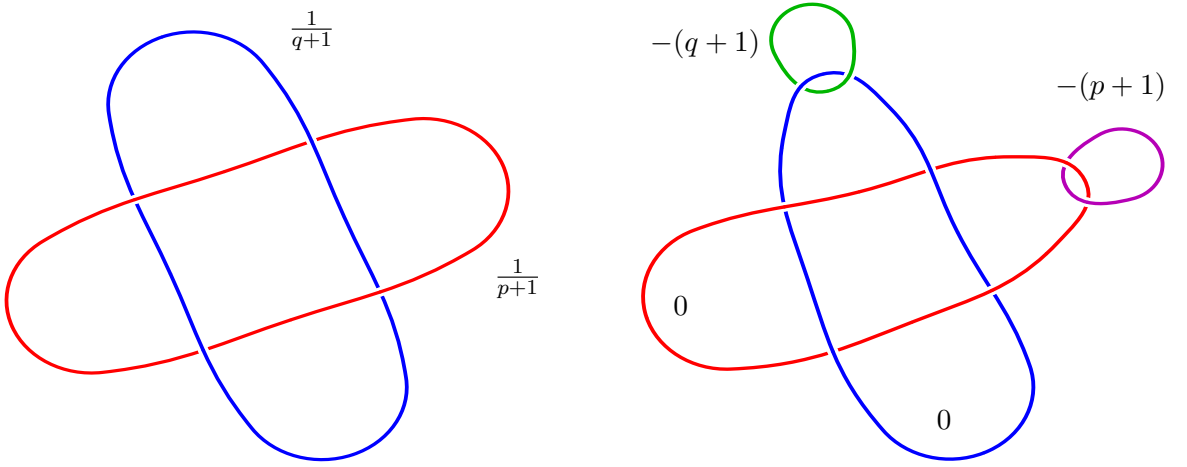
\begin{lemma}
 If the link $L$ and the diffeomorphism $\Phi$ are as previously defined then we have $p=q=0$ and then $\Phi$ also preserves meridians.   
\end{lemma}
\begin{proof}
 Proposition \ref{prop:complement} and Lemma \ref{lemma:1} imply that the $(1,1)$-surgery on $L$ is diffeomorphic to $S^3_{\frac{1}{p+1},\frac{1}{q+1}}(C_{2,4})$. Therefore, after applying two slam-dunks as in Figure \ref{Slamdunk}, see \cite{GS} for the precise definition, we have the following relation
 \[3=|H_1(S^3_{1,1}(L);\Z)|=\left|H_1\left(S^3_{\frac{1}{p+1},\frac{1}{q+1}}(C_{2,4});\Z\right)\right|=\left|\det\begin{pmatrix}
 -(p+1) & 1 & 0 & 0 \\
 1 & 0 & 2 & 0 \\
 0 & 2 & 0 & 1 \\
 0 & 0 & 1 & -(q+1)
\end{pmatrix}\right|=|4(p+1)(q+1)-1|\]
which implies a contradiction unless either $p=q=0$ or $p=q=-2$, but in the latter case we can repeat the same argument to obtain
\[3=|H_1(S^3_{-1,-1}(L);\Z)|=\left|H_1\left(S^3_{-\frac{1}{3},-\frac{1}{3}}(C_{2,4});\Z\right)\right|=\left|\det\begin{pmatrix}
 3 & 1 & 0 & 0 \\
 1 & 0 & 2 & 0 \\
 0 & 2 & 0 & 1 \\
 0 & 0 & 1 & 3
\end{pmatrix}\right|=35\] which is again a contradiction.
\end{proof}
Proving that $\Phi$ preserves both meridians and longitudes means exactly that the peripheral system of $L$ is invariant under $\Phi$ and this concludes the proof of Theorem \ref{teo:main}, as we mentioned before.

\section{Computation of the link Floer homology groups}
We compute $\widehat{HFL}(L)$ for any of the links in Theorem \ref{teo:main}. It can be observed from Table \ref{Table} that all such links have different homology and then their isotopy class is determined by $\widehat{HFL}$.
\begin{table}[ht]
\begin{center}
 \renewcommand{\arraystretch}{2}
  \begin{tabular}{ | c || c | c | c | }
    \hline
    $\widehat{HFL}_d(L)[s]$ & $s=-1$ & $s=0$ & $s=1$ \\ \hline
    $T_{2,3}\sqcup\bigcirc$ & $\F_{(-3)}\oplus\F_{(-2)}$ & $\F_{(-2)}\oplus\F_{(-1)}$ & $\F_{(-1)}\oplus\F_{(0)}$ \\ \hline
    $T_{2,-3}\sqcup\bigcirc$ & $\F_{(-1)}\oplus\F_{(0)}$ & $\F_{(0)}\oplus\F_{(1)}$ & $\F_{(1)}\oplus\F_{(2)}$ \\ \hline
    $4_1\sqcup\bigcirc$ & $\F_{(-2)}\oplus\F_{(-1)}$ & $\F_{(-1)}\oplus\F_{(0)}$ & $\F_{(0)}\oplus\F_{(1)}$ \\ \hline 
    $T_{2,2}\sqcup\bigcirc$ & $\F_{(-3)}\oplus\F_{(-2)}$ &  $\F_{(-2)}^2\oplus\F_{(-1)}^2$ & $\F_{(-1)}\oplus\F_{(0)}$ \\ \hline
    $T_{2,-2}\sqcup\bigcirc$ & $\F_{(-2)}\oplus\F_{(-1)}$ & $\F_{(-1)}^2\oplus\F_{(0)}^2$ & $\F_{(0)}\oplus\F_{(1)}$ \\ \hline
    $T_{2,4}'$ & $\F_{(-2)}^2$ & $\F^4_{(-1)}$ & $\F^2_{(0)}$ \\ \hline
    $T_{2,-4}'$ & $\F^2_{(-1)}$ & $\F^4_{(0)}$ & $\F^2_{(1)}$ \\ \hline 
    $T_{2,3;2,4}'$ & $\F^2_{(-3)}$ & $\F^4_{(-2)}\oplus\F_{(-1)}\oplus\F_{(0)}$ & $\F^2_{(-1)}$ \\ \hline 
    $T_{2,-3;2,-4}'$ & $\F^2_{(0)}$ & $\F_{(-1)}\oplus\F_{(0)}\oplus\F^4_{(1)}$ & $\F^2_{(2)}$ \\ \hline
  \end{tabular}
 \renewcommand{\arraystretch}{1}
\end{center}
\caption{The link Floer homology group, with collapsed Alexander grading, of every nearly fibered (multi-component) link with Seifert big genus equal to one.}
\label{Table}
\end{table} 
The homology groups of the mirror images and the split links are computed using Equations \eqref{eq:mirror} and \eqref{eq:disjoint}, starting from $\widehat{HFL}$ of the links in Theorem \ref{teo:classic} which can be found in \cite{Book}. 
The homology of $T_{2,4}'$ is also found in literature \cite{Multivariable}; hence, we are only left to discuss how to compute $\widehat{HFL}(T_{2,3;2,4}')$.

We start by computing the multi-graded link Floer homology group $\widehat{\textbf{HFL}}(T_{2,3;2,4})$, where this time the $(2,4)$-cable of the trefoil knot has components oriented in the same direction. Such a link is an $L$-space link and a cable of an $L$-space knot and then we can use the following result.
\begin{teo}[Gorsky-Hom \cite{GH}]
 Suppose that $K$ is an $L$-space knot in $S^3$; then the $(rm,rn)$-cable of $K$ is an $L$-space link for every $m,n$ coprime such that $\frac{n}{m}>2g(K)-1$ and $r\geq1$.

 Furthermore, the homology group $\widehat{\emph{\textbf{HFL}}}(K_{rm,rn})$ is completely determined by $m,n,r$ and the Alexander polynomial of $K$. 
\end{teo}
In our case we have $m=1,\:n=2,\:r=2$ and $K=T_{2,3}$. Since $2\geq1=2g(T_{2,3})-1$ the hypothesis of the theorem are satisfied. 
The formulae in \cite[Theorem 3]{GH} express $\widehat{\textbf{HFL}}(K_{rm,rn})$ in terms of the $h$-function of $K_{rm,rn}$, another link invariant introduced in \cite{GN} which as expected is also determined by the Alexander polynomial; moreover, in \cite[Lemma 4.5]{GH} it is also explained how to compute $h(K_{rm,rn})$ from $h(K)$. 

Finding the $h$-function of an $L$-space knot is very simple, given the structure of the knot Floer chain complex of this family of knots (see \cite{Qsurgery}), and in fact we can immediately write
\[h(T_{2,3})(k)=\left\{\begin{aligned}
    0\hspace{1cm}&k\geq1 \\
    1\hspace{1cm}&k=0 \\
    -k\hspace{1cm}&k\leq-1
\end{aligned}\right.\hspace{1cm}\text{ and }\hspace{1cm}h(T_{2,3;2,4})(k,k)=\left\{\begin{aligned}
    0\hspace{1cm}&k\geq2 \\
    1\hspace{1cm}&k=0,1 \\
    3\hspace{1cm}&k=-1 \\
    -2k\hspace{1cm}&k\leq-2
\end{aligned}\:.\right.\]
Plugging such values of $h(T_{2,3;2,4})$ into \cite[Theorem 3]{GH} yields 
\[\begin{aligned}\widehat{\textbf{HFL}}_d(T_{2,3;2,4})[s_1,s_2]&\cong\F_{(0)}[2,2]\oplus\F_{(-1)}[2,1]\oplus\F_{(-1)}[1,2]\oplus\F_{(-2)}[1,1]\oplus\F_{(-2)}[0,0]\oplus\F_{(-3)}[0,0]\:\oplus \\
\oplus\: &\F_{(-6)}[-1,-1]\oplus\F_{(-7)}[-2,-1]\oplus\F_{(-7)}[-1,-2]\oplus\F_{(-8)}[-2,-2]\:.\end{aligned}\]
Now, in order to compute the homology of $T_{2,3;2,4}'$ we need the following formula \cite[Proposition 11.4.2]{Book}, which relates the link Floer homology group of links where the orientation of one component differs:
suppose that $L$ has $n$-components, and let $L'$ be the
oriented link obtained from $L$ by reversing the orientation of its $i$-th component $L_i$. Then
\begin{equation}
 \label{eq:orientation}
 \widehat{\textbf{HFL}}_d(L')[s_1,...,s_n]\cong\widehat{\textbf{HFL}}_{d-2s_i+\lk(L_i,L\setminus L_i)}(L)[s_1,...,-s_i,...,s_n]\:.
\end{equation}
Applying Equation \eqref{eq:orientation} to $\widehat{\textbf{HFL}}(T_{2,3;2,4})$ we then obtain
\[\begin{aligned}\widehat{\textbf{HFL}}_d(T_{2,3;2,4}')[s_1,s_2]&\cong\F_{(-1)}[2,-1]\oplus\F_{(-1)}[-1,2]\oplus\F_{(0)}[0,0]\oplus\F_{(-1)}[0,0]\oplus\F_{(-2)}[2,-2]\:\oplus \\
\oplus\:&\F_{(-2)}[1,-1]\oplus\F_{(-2)}[-1,1]\oplus\F_{(-2)}[-2,2]\oplus\F_{(-3)}[-1,2]\oplus\F_{(-3)}[1,-2]\:;\end{aligned}\]
and it is now easy to check that, collapsing the Alexander grading as in Equation \eqref{eq:collapsed}, the group $\widehat{HFL}(T_{2,3;2,4}')$ is precisely the one appearing in Table \ref{Table}.

\end{document}